\documentclass[12pt]{amsart}

\usepackage{amssymb} 

\usepackage{amsmath} 

\usepackage{amsthm} 

\usepackage{url}

\numberwithin{equation}{section}
\theoremstyle{plain}
\newtheorem{theorem}{Theorem}[section]
\newtheorem{lemma}[theorem]{Lemma}
\newtheorem{proposition}[theorem]{Proposition}

\newtheorem{cor}{Corollary}

\theoremstyle{definition}
\newtheorem{definition}[theorem]{Definition}
\newtheorem{notation}[theorem]{Notation}

\theoremstyle{remark}

\newtheorem{example}[theorem]{Example}

\begin{document}

\title[Completeness in  Fourier-based  metrics]{On a completeness problem in  Fourier-based probability metrics in $\mathbb{R}^N$}

\author[Ma{\l}gorzata Stawiska]{Ma{\l}gorzata Stawiska}
\address{Mathematical Reviews, 416 Fourth St., Ann Arbor, MI 48103, USA}
\email{stawiska@umich.edu}

\date{\today}    
\maketitle

\begin{abstract}
We study completeness of the spaces $\mathcal{P}_s^=$ of  probability measures  in $\mathbb{R}^N$ which have equal (prescribed) moments  up to order  $s \in \mathbb{N}$, endowed with the metric $d_s(\mu,\nu)=\sup_{x \in \mathbb{R}^N\setminus 0}\frac{|\hat \mu(x)-\hat \nu(x)|}{|x|^s}$, where $\hat \mu$ is the characteristic function of $\mu$. We prove that the spaces  $(\mathcal{P}_s^=,d_s)$ are complete if  $s$ is even and construct  counterexamples to completeness for all odd $s$. This solves an open problem formulated by J. Carrillo and G. Toscani in 2007 (\cite{CT}).
\end{abstract}

\noindent \\

\section{Introduction} \label{s: Intro}

The possibility of introducing a metric on probability measures over a Polish space which metrizes their weak*-convergence is a topic with a long history. Many examples of such metrics are known (for collections of such examples, see e.g. \cite{Du02}, \cite{GSu}, \cite{V}), exploring properties of various quantities and objects associated to probability measures. The systematic  use of metrics based on the Fourier transforms of measures (and more generally, of tempered distibutions) was initiated in early 1950s by J. Deny (\cite{De50}, \cite{De51}). Developing and generalizing ideas in potential theory due to H. Cartan (\cite{Ca41}, \cite{Ca45}), he studied the metric   given for two probability measures $\mu, \nu$ on $\mathbb{R}^N$ by the $L^2$ norm of the Riesz potential of their difference: $d(\mu, \nu)=\bigl \|\frac{\hat{\mu}-\hat{\nu}}{|\cdot|^{s/2}}\bigr \|_2$, with $0<s<N$ (here $\hat \mu, \hat \nu$ are characteristic functions of respectively $\mu$ and $\nu$). Some of his arguments carry over to the case of $s=N$ (the logarithmic potential in $\mathbb{R}^N$)  and were  made explicit in \cite{CKL}. Deny's approach to energy of measures via Fourier transfrom was recently revisited in \cite{FZ}, giving rise to  the concept of the weak $\alpha$-Riesz energy of a signed Radon measure on $\mathbb{R}^N$ (\cite{FZ}, Definition 4.1) and  the theorem that the pre-Hilbert space of Radon measures on $\mathbb{R}^N$ with finite weak $\alpha$-Riesz energy is isometrically embedded into its completion, the Hilbert space  of real-valued tempered distributions with finite energy (\cite{FZ}, Theorem 5.1).\\
In 1990s a new class of metrics (known as Toscani metrics) dependent on a real positive parameter $s>0$,  motivated by applications to statistical physics and optimal transport theory,  was defined  by  using  Fourier transforms (characteristic functions) of (Borel) probability measures on $\mathbb{R}^N$,  namely, $d_s(\mu,\nu)=\sup_{x \in \mathbb{R}^N\setminus 0}\frac{|\hat \mu(x)-\hat \nu(x)|}{|x|^s} \quad s >0$ (see \cite{GTW}, \cite{TV}, \cite{CT} and below for conditions of finiteness and various properties, and \cite{TT} for recent applications  to wealth distribution). Completeness of various spaces of probability measures under these  Fourier-based metrics was considered in more detail in \cite{CT}. Independently, an $L^1$ version of a Fourier-based metric was proposed for probability measures on $\mathbb{R}$ in \cite{BG}: 
$d_r(\mu,\nu)=\int |\hat{\mu}(t)-\hat{\nu}(t)||t|^{-r-1}dt$. For $r \in (1,2)$ it was proved in \cite{BG} that the space of probability measures $\mu$ on $\mathbb{R}$ such that $\int|x|^r d\mu <+\infty$ and $\int xd\mu=c$ ($c \in \mathbb{R}$ fixed) is complete under the metric $d_r$. \\

The aim of our present note is to prove the following theorem, also concerning completeness of  certain spaces of probability measures endowed with the metric $d_s(\mu,\nu)=\sup_{x \in \mathbb{R}^N\setminus 0}\frac{|\hat \mu(x)-\hat \nu(x)|}{|x|^s}$.  

\begin{theorem}  Let $s \in \mathbb{N}$. (a)  For every   even number $s$ and every choice of numbers $M_\beta \in \mathbb{R}_+$ for  $|\beta| \leq  s$, the space of all probability measures    for which $\int_{\mathbb{R}^N}v^\beta d\mu(v)= M_\beta, \ |\beta| \leq  s$,  endowed with the metric $d_s$ is complete. (b) For every odd number $s$ there exist $M_\beta \in \mathbb{R}_+$ for  all $|\beta| \leq  s$ such that the space of all probability measures    for which $\int_{\mathbb{R}^N}v^\beta d\mu(v)= M_\beta, \ |\beta| \leq  s$, endowed with the metric $d_s$ is  not complete.
\end{theorem}

Completeness of such spaces was posed as an open problem in \cite{CT}.  The precise (separate) statements of both parts of the theorem (respectively, Theorems \ref{theorem: evenmom} and \ref{theorem: counterex}), all relevant definitions, auxiliary results, and finally corresponding  proofs and counterexamples  are presented in the next  sections. Note that since we make assumptions only on the moments of order up to an including $s$, not on any higher order, we cannot use a completeness result proved in \cite{CT}  as Proposition 2.7 in that paper. Note also that the authors of  \cite{TT} use a completeness result of the type of Proposition 2.7 in \cite{CT}. Among other things, they make and further generalize the observation that, if the initial datum of   the Fokker-Planck equation is a probability density on the positive half-line with mean one, then so is the solution at any subsequent time. In Theorem 3.7 in \cite{TT} under the assumption of boundedness of the metric $d_3$ between the initial datum and the equilibrium density they prove the convergence of the solution towards the equilibrium in the metric $d_2$. Anyway, they work exclusively on the positive half-line  and their results are only indirectly related to our problem at hand.  As far as other completeness results in literature are concerned,   we will  use (in our Example \ref{example:Cauchy}) a result from \cite{CK}, establishing  completeness of the set of all probability measures in the metric $d_1$ (originally stated in terms of Fourier transforms; for full generality, see Proposition 3.1.in that paper).  Finally, we should mention that yet another class of Fourier-based probability metrics was lately introduced by Cho (\cite{Cho15}), who also studied  relations of these metrics with absolute moments, but we  will not discuss these results here.  \\

\section{Preliminary material}

We will work in $\mathbb{R}^N, \ N\geq 1$.  Unless specified otherwise, $\|\cdot\|$ will denote the Euclidean norm in $\mathbb{R}^N$. By a ``probability measure on $\mathbb{R}^N$"  we mean a (Radon) probability measure on the Borel $\sigma$-algebra generated by the standard topology in $\mathbb{R}^N$.  We will need some standard definitions and results, which can be found e.g. in \cite{Du02}, Section 9:

\begin{definition} 
 Let $\mathcal{C}^b(\mathbb{R}^N)$ be the set of all bounded continuous real-valued functions on $\mathbb{R}^N$. We say that the probability measures $\mu_n$ converge (weakly*) to a probability measure  $\mu$ if and only if  for every $\varphi \in \mathcal{C}^b(\mathbb{R}^N)$, $\int \varphi d\mu_n \to \int \varphi d\mu$ as $n \to \infty$.
\end{definition}
\begin{notation} 
The function
\[
\hat \mu(x)=\int_{\mathbb{R}^N}e^{-i\langle x,v\rangle}d\mu(v), \ x \in \mathbb{R}^N, 
\]
where $\langle \cdot,\cdot \rangle$ denotes the standard scalar product in $\mathbb{R}^N$, is the Fourier transform (characteristic function) of $\mu$.
\end{notation}

Note that the function $\hat \mu$ is continuous and bounded on $\mathbb{R}^N$. Furthermore (cf. \cite{Du02}, Theorem 9.4.4), if $\int_{\mathbb{R}^N}v^\beta d\mu(v)$ is finite for a multi-index $\beta=(\beta_1,...,\beta_N)\in \mathbb{N}_0^N$, then $\hat \mu$ has continuous partial derivative  $D^\beta \hat \mu=\frac{\partial^{|\beta|}}{\partial x_1...\partial x_N} \hat \mu$ everywhere in $\mathbb{R}^N$, satisfying the formula $D^\beta\hat \mu (x)=\int_{\mathbb{R}^N}(iv)^\beta e^{-i\langle x,v\rangle}d\mu(v)$. Recall that $v^\beta=v_1^{\beta_1}...v_N^{\beta_N}$. The quantity $\int_{\mathbb{R}^N}v^\beta d\mu(v)$ is called the moment of $\mu$ of order $\beta$. If $\beta_1,..,\beta_N$ are even numbers, then the existence of $D^\beta\hat \mu (0)$ also implies the finiteness of the moment $\int_{\mathbb{R}^N}v^\beta d\mu(v)$ (cf. e.g. \cite{Sas}, Theorem 1.2.9). For $\beta_1+...+\beta_N$ odd  this is not the case: counterexamples in $\mathbb{R}$ were given in 1930s by A. Zygmund and A. Wintner (see \cite{Lu}, Section 2.3). As we will see, this dichotomy affects   completeness of the spaces $(\mathcal{P}_s^=,d_s)$ studied below, yielding a positive  answer when $s$ is even and a negative one when it is odd.\\

An important relation between the weak* convergence of probability measures and their Fourier transforms is given by the L\'evy continuity theorem:

\begin{theorem}\label{theorem:Levy} (cf. \cite{Du02}, Theorem 9.8.2): If $\mu_n, \ n=1,2,...$ are probability measures on $\mathbb{R}^N$ whose characteristic functions converge for all $x$ to some $g(x)$, where $g$ is continuous at $0$ along  each coordinate axis, then $\mu_n \to\mu$ weakly* to a probability measure $\mu$ with characteristic function $g$.
\end{theorem}

A significant part of the paper  \cite{CT} is devoted to the study of a metric based on Fourier transform, defined on suitable spaces of probability measures:

\begin{notation}(\cite{CT}, page 88) 
Fix  a  real number $s>0$. For any pair of probability measures $\mu, \nu$ on $\mathbb{R}^N$ we let 
\[
d_s(\mu,\nu)=\sup_{x \in \mathbb{R}^N\setminus 0}\frac{|\hat \mu(x)-\hat \nu(x)|}{|x|^s},
\]
where $\hat \mu(x)$ 
is the Fourier transform (characteristic function) of $\mu$.
\end{notation}

Clearly $d_s$ is nonnegative, symmetric in $\mu, \nu$, zero when $\mu=\nu$, and satisfies the triangle inequality. The inversion formula for Fourier transforms (cf. \cite{Du02}, Theorem 9.5.1) implies that $d_s(\mu,\nu)=0$ only if $\mu=\nu$. Thus $d_s$ defines a metric on a space of probability measures for which it is finite. 
A sufficient condition for finiteness of $d_s(\mu,\nu)$ was proved in \cite{CT}:

\begin{proposition}\label{prop: CTfiniteness}(\cite{CT}, Proposition 2.6): Let $s>0$ be given and let $\mathcal{P}_s$ denote the space of all probability measures on $\mathbb{R}^N$ with finite moments up to order $[s]$. The expression $d_s(\mu,\nu)$   is finite if $\mu, \nu$ have equal moments up to order $[s]$ if $s \not \in \mathbb{N}$ or up to order $s-1$ if $s \in \mathbb{N}$.
\end{proposition}

The proof given there is not conceptually difficult, but still quite involved. A much simpler way to prove this statement (without estimates of derivatives based on Lemma 2.5 of \cite{CT}) would be just to apply  Taylor's formula at $0 \in \mathbb{R}^N$ with remainder in Peano form (cf. \cite{Sch81a}, Ch. III, Section 5, Theorem 28) to the characteristic functions $\hat \mu$, $\hat \nu$, which by the assumptions on moments are $[s]$ times continuously differentiable in $\mathbb{R}^N$:
\[
\hat{\mu}(x) = \hat \mu(0)+\hat \mu'(0).x+...+\frac{\hat{\mu}^{([s])}(0)}{[s]!}.(x,...,x)+\alpha(x) \|x\|^{[s]},
\]
 where $\alpha$   tends to $0$ together with $x$. We omit the obvious details.

In what follows we will also use the  uniqueness property in Taylor's theorem:

\begin{theorem}(\cite{Sch81a}, Theorem III.7.38): Let $f$ be a $m$-fold differentiable mapping from an open set $\Omega \subset E$, where $E$ is an affine normed space, into a normed vector space $F$. If there exist $k$-linear continuous symmetric mappings $L_k$ of the space $E^k$ into $F$, $k=1,...,m$, and an element $L_0 \in F$ such that 
\[
f(a+h)=L_0+L_1h+\frac{L_2}{2!}h^2+...+\frac{L_m}{m!}h^m+\alpha\|h\|^m,
\]
where $\alpha=\alpha(h)$ tends to $0$ along with $h$, then necessarily
\[
L_k=f^{(k)}(a), \quad k=1,...,m
\]
and
\[
L_0=f(a).
\]
\end{theorem}

\section{Completeness in  the metric $d_s$: positive results ($s$ even)}\label{s: positive}

Some subspaces of $\mathcal{P}_s$ are complete with respect to the metric $d_s$. An example is presented in  \cite{CT}: given $s,a>0$, let us denote by $X_{s,a,M}$ the set of probability measures $\mu \in \mathcal{P}_{s+a}(\mathbb{R}^N$ such that $\int_{\mathbb{R}^N}v^\beta d\mu(v)=M_\beta \in \mathbb{R}_+$ for all multi-indices $|\beta|\leq [s]$ with $M_\beta$ fixed numbers and $\int_{\mathbb{R}^N}v^{s+a} d\mu(v)\leq M_{s+a}\in \mathbb{R}_+$, where the set of all $M_\beta$ and $M_{s+a}$ is denoted simply by $M$.  By Proposition 2.7 in  \cite{CT}, the set $X_{s,a,M}$ endowed with the distance $d_s$ is a complete metric space. But, as noticed in Remark 2.8 of \cite{CT}, the proof of this Proposition given in that paper does not establish the completeness of the set $\mathcal{P}_s^=$ of probability measures $\mu \in \mathcal{P}_s(\mathbb{R}^N)$ with $s \in \mathbb{N}$, such that $\int_{\mathbb{R}^N}v^\beta d\mu(v) = M_\beta\in \mathbb{R}_+$ for all multi-indices $|\beta|\leq s$ with $M_\beta$ given, endowed with the distance $d_s$. Consequently, an open problem was formulated within the same remark as follows: ``It would be nice to prove or rather disprove such statement at least for the $d_2$ distance."   Below we are going to solve this problem, proving completeness of $(\mathcal{P}_s^=,d_s)$ for $s$ even and disproving it for $s$ odd.   The behavior of spaces $(\mathcal{P}_2^=,d_2)$ is thus different from that  conjectured in  \cite{CT}.

\begin{notation}\label{notation:Ps}
Fix  a  natural number $s$ and  numbers $M_\beta \in \mathbb{R}_+$ for all $|\beta| \leq  s$. We let $\mathcal{P}_s^=$ denote the space of all probability measures in $\mathcal{P}_s$ which have equal moments up to order $s$, i.e.,   for which $\int_{\mathbb{R}^N}v^\beta d\mu(v)= M_\beta$ for all $|\beta| \leq  s$.
\end{notation} 

Strictly speaking, the space $\mathcal{P}_s^=$ depends also on the  numbers $M_\beta, |\beta| \leq  s$, but we suppress this dependence in the notation.  Instead, in what follows we will specify whether we use arbitrary $M_\beta$ or concrete ones.

First we will prove the following necessary condition:

\begin{lemma} \label{lemma:diff0} Fix an $s \in \mathbb{N}$ and arbitrary $M_\beta \in \mathbb{R}_+$ for all $|\beta| \leq  s$. Let  $\mathcal{P}_s^=$  be defined as in Notation \ref{notation:Ps} and let $\mu_n \in \mathcal{P}_s^=$ be a convergent sequence with respect to the metric $d_s$. Then the characteristic function $\hat{\mu}$ of the limit measure $\mu$ is $s$- fold differentiable at $0$, with $D^\beta\hat \mu(0)=M_\beta$ for all $|\beta|=s$.
\end{lemma}

\begin{proof}

  Let $L_0=1$, $L_k = D^{(k)}\hat \mu_n(0), \quad k=1,...,s$. Recall that 

\[
D^{(k)}\hat \mu_n(0)(x_1,...,x_k)=\sum_j\frac{{\partial}^k f}{\partial x_{j_1}...\partial x_{j_k}}(0)x_{1,j_1}x_{2,j_2}...x_{k,j_k},
\]
where $j=(j_1,...j_k)$, $|j|=j_1+...+j_k=k$. From the assumption of equality of moments of corresponding orders up to $s$ for all $\mu_n$, the mapping $L_k$ is the same $k$-linear continuous symmetric mapping for all $n \in \mathbb{N}$. We already know that $\hat \mu (0)=1$. Let $\varepsilon >0$ be given. We estimate 

\[
\frac{\|\hat \mu(x) -(L_0+L_1.x+...+\frac{L_s}{s!}.(x,...,x))\|}{\|x\|^s}
\]
from above by 
\[
\frac{\|\hat \mu_n(x) -(L_0+L_1.x+...+\frac{L_s}{s!}.(x,...,x))\|}{\|x\|^s}+\frac{\|\hat \mu_n (x)-\hat \mu(x)\|}{\|x\|^s}, 
\]
where $n$ is fixed so that $\frac{\|\hat \mu_n (x)-\hat \mu(x)\|}{\|x\|^s}<\varepsilon/2$ (by $d_s$-convergence of $\mu_n$ to $\mu$  such $n$ exist). By  Taylor's formula we can pick a  $\delta >0$ such that $\frac{\|\hat \mu_n(x) -(L_0+L_1.x+...+\frac{L_s}{s!}.(x,...,x))\|}{\|x\|^s} < \varepsilon/2$ as $\|x\|<\delta$. This proves that $\|\hat \mu(x) -L_0+L_1.x+...+\frac{L_s}{s!}.(x,...,x)\|=o(\|x\|^s)$. By the uniqueness in  Taylor's theorem, $L_k=D^{(k)}\hat \mu(0)$ for all  $k=1,...,s$. This implies the differentiability of $\hat \mu$  at $0$ and also the   equality of all partial derivatives of $\hat \mu$  at $0$ with the derivatives of corresponding orders of all  $\hat \mu_n$ at $0$.

\end{proof}

Lemma \ref{lemma:diff0} allows us to give an example of a sequence of probability measures on $\mathbb{R}$ which is not a Cauchy sequence with respect to the metric $d_1$. 

\begin{example} \label{example:Cauchy}
Let $\mu_n$ be the measure with density $p_n(x)=c_n(1+x^2)^{-(1+1/n)}, \ n=1,2,...$. The constants $c_n$ are \[c_n=\frac{\Gamma(1+\frac{1}{n})}{\sqrt{\pi} \Gamma (\frac{1}{2}+\frac{1}{n})},\] with $\Gamma$ denoting the standard (Euler) gamma function. The sequence $\mu_n$ converges weakly* to the Cauchy distribution on $\mathbb{R}$ (by e.g. Proposition 1.6.6 in \cite{Sas}). By Theorem 6.13 in \cite{We}, the characteristic function of $\mu_n$ is $\varphi_n(t)=c_n\frac{2^{1/n}}{\Gamma(1+1/n)}|t|^{1/2+1/n}K_{1/2+1/n}(|t|), t \in \mathbb{R}$, where $K_\nu$ is the modified Bessel function of the second kind of order $\nu$. Let $\mathcal{K}_1$ denote the space of characteristic functions $\varphi$ such that $\sup_{t \in \mathbb{R}} \frac{|\varphi(t)-1|}{|t|}<\infty$. We have $\varphi_n \in \mathcal{K}_1$.  Suppose $\mu_n$ is a Cauchy sequence with respect to $d_1$. By Proposition 3.10 in \cite{CK}, the space $\mathcal{K}_1$ is complete with the metric $\|\varphi-\psi\|_1=\sup_{t \in \mathbb{R}} \frac{|\varphi(t)-\psi(t)|}{|t|}$, so $\varphi_n$ converge in $\|\cdot\|_1$ to a characteristic function $\varphi$ of some probability measure $\mu$ or equivalently, $\mu_n$ converge to $\mu$ in the metric $d_1$. By Lemma \ref{lemma:diff0}, $\varphi$ is differentiable at $0$. But $\varphi_n(t) \to \exp (-|t|)$ as $n \to \infty$, which is a contradiction.
\end{example}

We will also need the following auxiliary result:

\begin{lemma} \label{lemma:unilip} For $s \geq 2$ consider a sequence of measures $\{\mu_n\}_{n \in \mathbb{N}} \subset \mathcal{P}_s^=$. Then  the characteristic functions $\hat \mu_n$ are uniformly Lipschitz in $\mathbb{R}^N$. 
\end{lemma} 
\begin{proof}

It is enough to show that all derivatives $D^\beta \hat \mu_n(x), \ |\beta|=1$ are uniformly bounded. 
Fix a multi-index $\beta$ with $|\beta|=1$. Then $|D^\beta \hat \mu_n(x)| \leq \int_{\mathbb{R}^N}|v^\beta|d\mu_n(v) \leq \int_{\mathbb{R}^N} \|v\|d\mu_n(v) \leq 1 +\int_{\|v\|>1}\|v\|^2d\mu_n(v)$. From the assumption that $\{\mu_n\}_{n \in \mathbb{N}} \subset \mathcal{P}_s^=$ with $s \geq 2$ (existence and equality of all corresponding moments of order $2$ for all $\mu_n$) we infer the existence of common bound for all  $|D^\beta \hat \mu_n(x)|, x \in \mathbb{R}^N$.

\end{proof}

Now we will prove a positive result concerning completeness of the metric spaces $(\mathcal{P}_s^=, d_s)$:

\begin{theorem}\label{theorem: evenmom}
Let $s$ be an even number. Fix the sequence of numbers $M_\beta \in \mathbb{R}_+, \ |\beta| \leq  s$ and consider the corresponding space $\mathcal{P}_s^=$.  Let the sequence of measures $\{\mu_n\}_{n \in \mathbb{N}} \subset \mathcal{P}_s^=$ be a Cauchy sequence in the metric  $d_s$. Then there exists a probability measure  $\mu \in \mathcal{P}_s^=$ such that $d_s(\mu_n,\mu) \to 0$ as $n \to \infty$.   
\end{theorem}

\begin{proof}
Since  $\{\mu_n\}_{n \in \mathbb{N}} \subset \mathcal{P}_s^=$ is a Cauchy sequence with respect to the metric $d_s$, then  for every $\varepsilon >0$ there is an $n_\varepsilon \in \mathbb{N}$ such that for every $n,m > n_\varepsilon$ one has the inequality $\sup_{x \in \mathbb{R}^N \setminus 0}\frac{|\hat\mu_n(x)-\hat\mu_m(x)|}{|x|^s}<\varepsilon$. We will first show (similarly to  Proposition  2.7 in \cite{CT}, but again without their Lemma 2.5) that $\mu_n$ converges weakly* to a probability measure $\mu$.   To do this, observe that for any fixed $x\in \mathbb{R}^N$ the sequence  $\{\hat \mu_n(x)\}_{n \in \mathbb{N}}$ is a Cauchy sequence in $\mathbb{C}$ (for $x=0$ the sequence is constant, of value $1$). Hence the sequence of Fourier transforms $\{\hat \mu_n\}$ converges pointwise to some function $g$ on $\mathbb{R}^N$ with $g(0)=1$. Since $s \geq 2$, Lemma \ref{lemma:unilip}  implies that the functions $\hat \mu_n$ are uniformly Lipschitz  in $\mathbb{R}^N$. By Ascoli 's theorem (cf. \cite{Sch81b}, Chapter VII, Section 6, Theorem 48, Corollary 1 and preceding examples), $\hat\mu_n$ converge uniformly on compact subsets to a continuous function $g$.  By L\'evy continuity theorem  $g$ is a Fourier transform of a probability measure $\mu$ and $\mu_n \to \mu$ weakly*  as $n \to \infty$. Similarly as in Lemma \ref{lemma:unilip} it can be shown for every $j=1,...,s-1$ that the maps $D^\beta\mu_n, |\beta|=j$ are uniformly Lip\-schitz, so Ascoli's theorem can be applied repeatedly to show that $\hat \mu$ has all derivatives up to order $s-1$ at $0$ and they are equal to the corresponding derivatives of $\hat \mu_n$. In particular, $d_s(\mu_n,\mu)$ is finite for all $n$.\\

We will now show that $d_s(\mu_n,\mu) \to 0$. Note that our assumptions imply that the sequence of complex-valued functions $f_n(x)=\frac{\hat\mu_n(x)-\hat\mu(x)}{|x|^s}$ is a Cauchy sequence in the space $\mathcal{C}^b(\mathbb{R}^N \setminus 0)$ of continuous bounded (complex-valued) functions in $\mathbb{R}^N \setminus 0$ endowed with the supremum norm. This is a complete metric space, so there exists a function $f \in \mathcal{C}^b(\mathbb{R}^N \setminus 0)$ such that $f_n$ converge to $f$ uniformly in $\mathbb{R}^N \setminus 0$. Hence for every $x \ne 0$ we have $|\hat \mu_n(x) -\hat \mu (x)| \to |x|^sf(x)$ as $n \to \infty$. The convergence of $\hat \mu_n$ to $\hat \mu$  implies that for every $x\ne 0$ one has $|x|^s|f(x)|=0$, and so $f(x)=0$ for every $x\ne 0$. This proves the claim that $d_s(\mu_n,\mu) \to 0$ as $n \to \infty$.\\

Finally, since the sequence $\mu_n$ converges in $d_s$, Lemma \ref{lemma:diff0} yields that $\hat \mu$ is $s$-fold differentiable at $0$ and that the partial derivatives at $0$ of all orders $\beta, \ |\beta|\leq s$, are equal to corresponding derivatives  of $\hat \mu_n$ at $0$, that is, to the numbers $M_\beta$. Let $\beta =(2\gamma_1,...,2\gamma_N)$ with $|\beta| \leq  s$. Then  the moment of order $\beta$ of $\mu$ exists and is equal to $M_\beta$.  In particular, $\mu$  has  all moments of orders $(s,0,...,0)$, $(0,s,0,...,0)$,..., $(0,...0,s)$ and hence also of orders $(s',0,...,0)$, $(0,s',0,...,0)$,..., $(0,...0,s')$ for all $1 \leq s' < s$, which are equal to the corresponding numbers  $M_\beta$. It remains to prove the existence of moments of orders $\beta=(\beta_1,...,\beta_N)$ with $|\beta|=s$    such that  not all $\beta_i$ are necessarily even and at least two of them are nonzero.  Without loss of generality we can assume that $\beta_1 \neq 0, \beta_2\neq 0$. If $\beta_1+\beta_2=s$, then 
$\int_{\mathbb{R}^N} |x_1|^{\beta_1}|x_2|^{\beta_2}|x_3|^{\beta_3}...|x_N|^{\beta_N}=  \int_{\mathbb{R}^N} |x_1|^{\beta_1}|x_2|^{\beta_2} d\mu \leq \int_{|x_1| \leq |x_2|}|x_2|^s +\int_{|x_2| < |x_1|}|x_1|^s < \infty$. If $\beta_1+\beta_2<s$, then by the same argument $|x_1|^{\beta_1}|x_2|^{\beta_2}  \in L^p(\mu)$, where $p =s/(\beta_1+\beta_2) >1$. The exponent  $q=s/(\beta_3+...+\beta_N)$ satisfies $1/p +1/q=1$.  If there is only one coordinate $\beta_i \neq 0, i=3,...N$, then $\int_{\mathbb{R}^N} (|x_i|^{\beta_i})^q= \int_{\mathbb{R}^N} |x_i|^s <\infty$. Suppose that for any combination of natural $\beta_{i_1} \geq 1,...,\beta_{i_K} \geq 1$  with $3 \leq i_1 <...<i_K \leq N$ and $\beta_1+\beta_2+\beta_{i_1}+...+\beta_{i_K}=s$ we have $\int_{\mathbb{R}^N} (|x_{i_1}|^{\beta_{i_1}}...|x_{i_K}|^{\beta_{i_K}})^q d \mu < \infty$. Inductively in $K$, take $\beta_{i_1} \geq 1,...,\beta_{i_{K+1}} \geq 1$  with $3 \leq i_1 <...<i_{K+1} \leq N$ and $\beta_1+\beta_2+\beta_{i_1}+...+\beta_{i_{K+1}}=s$. Then \begin{multline*}
\int_{\mathbb{R}^N} (|x_{i_1}|^{\beta_{i_1}}...|x_{i_{K+1}}|^{\beta_{i_{K+1}}})^q d \mu \leq\\ \int_{|x_{i_1}| \leq |x_{i_{K+1}}|} (|x_{i_2}|^{\beta_{i_2}}...|x_{i_{K+1}}|^{\beta_{i_1}+\beta_{i_{K+1}}})^q d \mu \\+ \int_{|x_{i_1}| > |x_{i_{K+1}}|} (|x_{i_1}|^{\beta_{i_1}+\beta_{i_{K+1}}}...|x_{i_{K}}|^{\beta_{i_K}})^q d \mu \\<\infty.
\end{multline*}
In particular, $|x_3|^{\beta_3}...|x_N|^{\beta_N} \in L^q(\mu)$. By H\"older inequality $\int x^\beta d\mu < \infty$. It follows that   $\int x^\beta d\mu=M_\beta$ for all $|\beta| \leq s$, so $\mu \in \mathcal{P}_s^=$. 
\end{proof}

\section{Completeness in  the metric $d_s$: negative results ($s$ odd)}\label{s: negative}

In this section we take $N=1$. We disprove completeness of the spaces $(\mathcal{P}_s^=,d_s)$ with $s$ odd integer by constructing, for each $s$, a space $\mathcal{P}_s^=$,  a sequence of probability measures $\mu_n$ in $\mathcal{P}_s^=$ and a measure $\mu$ such that $d_s(\mu_n,\mu) \to 0$ as $n \to \infty$, but $\mu \not \in \mathcal{P}_s^=$.
For this purpose we will modify the known example of a probability measure without the first moment but with characteristic function $\varphi$ admitting $\varphi'(0)$, due to A. Wintner (see below).\\

\subsection{Case $s=1$} 

 Consider the probability density function 
\[
p(x)=
\begin{cases} 
0 & \text{if } |x|<2,\\

\frac{C}{x^2\log|x|}  & \text{if } |x| \geq 2.
\end{cases}
\]

We have $2C=(\int_2^{\infty} \frac{1}{x^2\log x}dx)^{-1}=\frac{1}{E_1(\log 2)}$, where $E_1$ is the exponential integral function (\cite{AS}, Section 5) \[E_1(z)=\int_z^\infty\frac{e^{-u}}{u}du, \  |\arg z| < \pi.\] The corresponding characteristic function is 
\[
\varphi(t)=2C\int_2^\infty\frac{\cos tx}{x^2 \log x}dx.
\]

 It is well known (\cite{Lu}, Section 2.3) that $\varphi$ is differentiable at $0$ (with derivative equal to $0$) but  $\mu$ does not have the  first moment.

Our goal is to  find a sequence of probability measures $\mu_n$ with the Fourier transforms $\varphi_n$ such that each $\mu_n$ has the first moment equal to $0$ and
\[
\lim_{n\to\infty}\sup_{t\ne 0}\bigl |\frac{\varphi_n(t)-\varphi(t)}{t}\bigr |=0
\]
Let $\mathcal{P}_1^=$ denote the space of probability measures in $\mathbb{R}$ with the first moment equal to $0$. Then $\mu_n \in \mathcal{P}_1^=$, $d_1(\mu_n,\mu) \to 0$ (in particular, the sequence $\{\mu_n\}$ is Cauchy), but $\mu \not  \in \mathcal{P}_1^=$. 

We will  construct the sequence $\mu_n, n\geq 3$ with the following properties:
First,  each 
$\mu_n$ is symmetric: $\mu_n((-b,-a))=\mu_n((a,b))$ for all $0\le a <b$.   Second, $\mu_n$ is a sum of two nonnegative finite measures $\sigma_n +\gamma_{n}$, where    $\sigma_n$ is the truncation of $\mu$ to the interval $[-n,n]$: $\sigma_n(X)=\mu(X\cap [-n,n])$ for any Borel measurable set $X\subset \mathbb{R}$. In more detail, 
 $\sigma_n$ has the following density function: 
 \begin{eqnarray*}
p_n(x)=
\begin{cases}
0, \textrm{ for } |x|<2,\\
\frac{C}{x^2\log |x|}, \textrm{ for }n\ge |x|\ge 2,\\
0 , \textrm{ for } |x|>n.
\end{cases}
\end{eqnarray*}
 We take $\gamma_{n}$ to have the   density function 
\begin{eqnarray}\notag
r_{n}(x)=
\begin{cases}
0, \textrm{ for } |x|<n \\
\frac{\omega_n}{|x|^3} \textrm{ for }  |x| \ge n,
\end{cases}\\
\omega_n=\frac{C\int_{n}^{\infty}\frac{1}{x^2\log x}dx}{\int _n^\infty{\frac{1}{x^3}dx}}=(2n^2 C)\int_{n}^{\infty}\frac{1}{x^2\log x}dx.\label{defomegn}
\end{eqnarray}  
Then each $\mu_n$ has the first moment equal to $0$. 

Let $\tau_n=\mu-\alpha_n$.  
Denote by $f_n(t), g_n(t), h_n(t)$ the Fourier transforms of $\sigma_n$, $\gamma_n$ and $\tau_n$  respectively.  Then $\hat {\mu_n}(t)=\varphi_n(t)=f_n(t)+g_n(t)$ and $\hat \mu (t)= \varphi(t)=f_n(t)+h_n(t)$. 
Observe that 
\begin{eqnarray*}
h_n(0)=2C\int_n^{\infty} \frac{1}{x^2 \log x}dx=g_n(0)=\omega_n\frac{1}{n^2}.
\end{eqnarray*}
Thus

\begin{multline}
|\varphi_n(t)-\varphi(t)|=|g_n(t)-h_n(t)|=|(g_n(t)-g_n(0))-(h_n(t)-h_n(0))|\\
\le |g_n(0)-g_n(t)|+|h_n(0)-h_n(t)|.
\end{multline}

To  prove the convergence of $\mu_n$ to $\mu$ in the metric $d_1$ 
it is thus enough to show that
\begin{eqnarray}\label{aproxprop1}
\lim_{n\to\infty}\sup_{t\ne 0}\bigl |\frac{h_n(0)-h_n(t)}{2Ct}\bigr |=0, \ \lim_{n\to\infty}\sup_{t\ne 0}\bigl |\frac{g_n(0)-g_n(t)}{2Ct}\bigr |=0.
\end{eqnarray}

We will now prove these equalities. Our arguments will proceed similarly  to those  in the proof of existence of $\varphi'(0)$ in \cite{Lu}, based on the  estimates $$0\le 1-\cos z\le \min(2,z^2)$$ for any real $z$. We will need other estimates as well.

\begin{lemma}\label{lemma:integrals}
The following inequalities hold:
\begin{eqnarray}\notag
\frac{x}{\log x}\bigg|_n^y\le \int_n^y\frac{dx}{\log x}\le \frac{1}{1-1/\log n} \frac{x}{\log x}\bigg|_n^y\le \frac{1}{1-1/\log n} \frac{y}{\log y}
\end{eqnarray}
for $y>n\ge 3$..
\end{lemma}
\begin{proof}

Integration by parts with $u'=1, v=\frac{1}{\log x}$  gives 
\begin{eqnarray*}
\int_n^y \frac{dx}{\log x}=\frac{x}{\log x}\bigg|_n^y +\int_n^y\frac{dx}{(\log x)^2}=
\frac{x}{\log x}\bigg|_n^y +\frac{x}{(\log x)^2}\bigg|_n^y +2 \int_n^y\frac{dx}{(\log x)^3}.
\end{eqnarray*}
From 

$$\int_n^y\frac{dx}{(\log x)^2}\le \frac{1}{\log n} \int_n^y \frac{dx}{\log x} $$ 
we deduce
$$\int_n^y \frac{dx}{\log x}\le \frac{x}{\log x}\bigg|_n^y  +\frac{1}{\log n}\int_n^y \frac{dx}{\log x}$$

and 
\begin{eqnarray*}
\frac{x}{\log x}\bigg|_n^y\le \int_n^y\frac{dx}{\log x}\le \frac{1}{1-1/\log n} \frac{x}{\log x}\bigg|_n^y\le \frac{1}{1-1/\log n} \frac{y}{\log y} 
\end{eqnarray*}
 for  $y>n\ge 3$ as claimed.
\end{proof}

\begin{lemma} For $n\ge 3$,
\[
\frac{1}{1+1/\log n} \frac{1}{n\log n}<\int_n^\infty \frac{dx}{x^2\log x}<\frac{1}{n\log n}.
\]
\end{lemma}
\begin{proof} These inequalities follow from the estimates for the function $E_1$ (\cite{AS}, formula 5.1.19):
\begin{equation}\label{equation:intexp}
\frac{1}{x+1} < e^xE_1(x) \leq \frac{1}{x}, \ x >0.
\end{equation}

\end{proof}

\begin{cor}
\[
\frac{1}{1+1/\log n} \frac{2nC}{\log n}<\omega_n<\frac{2nC}{\log n}.
\]
\end{cor}


  \begin{theorem}\label{theorem:one} Let, $\mu, \mu_n, \ n \geq 3$ be probability measures defined as above, with characteristic functions $\varphi, \varphi_n$. Then
\[
\lim_{n\to\infty}\sup_{t\ne 0}\bigl |\frac{\varphi_n(t)-\varphi(t)}{t}\bigr |=0.
\]

\end{theorem}

\begin{proof}

To show the first equality in \ref{aproxprop1}, 
observe that 
\begin{eqnarray*}
0\le \frac{h_n(0)-h_n(t)}{2C}=\int_n^{\infty}\frac{1-\cos tx}{x^2\log x}dx.
\end{eqnarray*}
Since $h_n(t)$ is an even function, it is enough to assume that $t>0$.  

First assume that $t\ge 1/n$.   Then for $x\ge n$ we  use the  inequality $0\le 1-\cos tx\le 2$.  Therefore 
\begin{eqnarray*}
0\le \frac{h_n(0)-h_n(t)}{2C}=\int_n^{\infty}\frac{1-\cos tx}{x^2\log x}dx\le 2\int_n^\infty \frac{1}{x^2\log x}  \leq \frac{2}{n\log n}.
\end{eqnarray*}
Thus
\begin{eqnarray*}
0\le \frac{h_n(0)-h_n(t)}{2Ct}\le \frac{2}{\log n} \textrm{ for } t\ge 1/n.
\end{eqnarray*}
Next assume that $0<t<1/n$  
 and for $tx<1$ use the inequality $1-\cos tx\le (tx)^2$.   Then 
\begin{eqnarray*}
\frac{h_n(0)-h_n(t)}{2C}= \int _n^{1/t} \frac{1-\cos tx}{x^2\log x}dx + \int _{1/t}^{\infty} \frac{1-\cos tx}{x^2\log x}dx\le \\
t^2\int_{n}^{1/t} \frac{dx}{\log x}+2\int_{1/t}^{\infty} \frac{dx}{x^2\log x}\le 
t^2\int_{n}^{1/t} \frac{dx}{\log x}+2\frac{1}{\log (1/t)}\int_{1/t}^\infty \frac{1}{x^2}dx\le\\
\frac{1}{(1-1/\log n)}t^2 \frac{1/t}{\log (1/t)}+2t/(\log1/t)\le  (2+\frac{1}{(1-1/\log n)})t(1/\log n),
\end{eqnarray*}
for $0<t<1/n$. 
This establishes the first equality in  \eqref{aproxprop1}. 

We now establish  the second equality in  \eqref{aproxprop1}.  Recall that
\begin{eqnarray*}
0\le \frac{g_n(0)-g_n(t)}{2C}=\frac{\omega_n}{2C}\int_n^\infty \frac{1 -\cos tx}{x^3}dx
\end{eqnarray*}
As with $h_n$, we can assume that $t>0$.   Consider two cases.   First $t\ge 1/n$.  Then we use the inequality $1-\cos tx\le 2$ and the bound for $\omega_n$ to deduce
\begin{eqnarray*} 
0<\frac{g_n(0)-g_n(t)}{2Ct}\le n\frac{g_n(0)-g_n(t)}{2C}\le n\frac{\omega_n}{2C} 2\int_n^\infty \frac{dx}{x^3}< \frac{1}{\log n}.
\end{eqnarray*}

Next we consider the case where $0<t<1/n$.  Here for $tx<1$ we will use the inequality $1-\cos tx\le (tx)^2$.  Hence

\begin{multline*}
0<\frac{g_n(0)-g_n(t)}{2C}= \frac{\omega_n}{2C}\big(\int _n^{1/t} \frac{1-\cos tx}{x^3}dx + \int _{1/t}^{\infty} \frac{1-\cos tx}{x^3}dx\big)\le \\
\frac{n}{\log n}\big(t^2\int_n^{1/t}\frac{dx}{x} +2\int _{1/t}\frac{dx}{x^3}\big)=\frac{t}{\log n}\big(-(nt)\log (nt) +nt\big)
\end{multline*} 

As $0<nt<1$, we deduce the inequality $-(nt)\log(nt)\le 1/ e$ (where $e$ is the natural logarithm base). Thus we get the inequality
\begin{eqnarray*}
0<\frac{g_n(0)-g_n(t)}{2Ct}<\frac{1+1/e}{\log n} \textrm{ for }0<t<1/n. 
\end{eqnarray*}
This shows the second equality in  \eqref{aproxprop1}.

\end{proof}

\subsection{Case $s=2k+1, k \geq 1$.}\label{subsection:greater}
Fix $s=2k+1$, where $k\ge 1$ is an integer. It will be more convenient to denote various quantities as dependent on $k$. Define $\mu_k$ as the probability measure in $\mathbb{R}$ with density
\begin{eqnarray*}
p_k(x)=
\begin{cases}
0, \textrm{ for } |x|<2,\\
\frac{C_k}{x^{2+2k}\log |x|}, \textrm{ for } |x|\ge 2,
\end{cases}\\
C_{k}=\bigl(2\int_2^{\infty} \frac{1}{x^{2+2k}\log x}dx\bigr)^{-1}=\biggl(\frac{2}{2k+1}E_1((2k+1)\log 2)\biggr)^{-1} .
\end{eqnarray*}
Then  $\mu_s$ has all the moments  up to order $2k$ but does not have the moment of order $2k+1=s$.   Concretely,

\[
\  M_{2\ell}=2\int_2^\infty\frac{C_kx^{2\ell}}{x^{2+2k}\log x}dx =\frac{2C_k}{2(k-\ell)+1}E_1\bigl((2(k-\ell)+1)\log 2\bigr)
\]

and $M_{2\ell-1}=0,\ \ell=1,...,k$. We will also define $M_0:=1$. and use the sequence $M_0,...,M_s$ to define the space $\mathcal{P}_s^=$. 
On the other hand, the characteristic function 
\begin{eqnarray*}
\varphi_k(t)=\int_{-\infty}^{\infty}e^{-i tx}d\mu_k(x).
\end{eqnarray*}
 has  derivatives up to order $2k+1=s$ at $0$ and $\varphi_k^{(s)}(0)=0$.
Now, for each integer $n\ge N_k\in\mathbb{N}$ we will construct a probability measure $\mu_{n,k}$, such that  $\mu_{n,k}$ has $M_1, M_2,..., M_{2k}$ as its moments of orders up to $2k$ and $0$ as its moment of order $s=2k+1$.
    For   the Fourier transforms $\varphi_{n,k}$ of $\mu_{n,k}$ we will show that 
  \begin{eqnarray}\label{aproxpropk}
\lim_{n\to\infty}\sup_{t\ne 0}\bigl|\frac{\varphi_{n,k}(t)-\varphi_k(t)}{t^{2k+1}}\bigr|=0.
\end{eqnarray}
Fix a small positive $\varepsilon=\varepsilon(k)>0$ that will be specified 
later.  
The measures $\mu_{n,k}$ are constructed as follows:
First, $\mu_{n,k}$ is symmetric $\mu_{n,k}((-b,-a))=\mu_{n,k}((a,b))$ for all $0\le a <b$.  Hence the odd-order moments of $\mu_{n,k}$ (up to and including $s$) are zero. Second, $\mu_{n,k}$ is a sum of two nonnegative measures $\sigma_{n,k} +\gamma_{n,k}$, where    $\sigma_{n,k}$ is the truncation of $\mu_k$ to the interval $[-n,n]$: $\sigma_{n,k}(X)=\mu_k(X\cap [-n,n])$ for any Borel measurable set $X\subset \mathbb{R}$.

Specifically, $\sigma_{n,k}$ has the  density function  
 \begin{eqnarray*}
p_{n,k}(x)=
\begin{cases}
0, \textrm{ for } |x|<2,\\
\frac{C_k}{x^{2(k+1)}\log |x|}, \textrm{ for }n\ge |x|\ge 2,\\
0 , \textrm{ for } |x|>n 
\end{cases}
\end{eqnarray*}
 and we take $\gamma_{n,k}$ to have the following density function:  
\begin{eqnarray}\notag
r_{n,k}(x)=
\begin{cases}
0, \textrm{ for } |x|<n \\
\frac{C_k\omega_{n,j}}{|x|^{2(k+1)+e}} \textrm{ for }  |x|\in [2^{j-1}n,2^jn),j=1,...,k,\\
\frac{C_k\omega_{n,k+1}}{|x|^{2(k+1)+e}} \textrm{ for }  |x| \ge 2^{k}n.
\end{cases}
\end{eqnarray}  

We need to prove the following: 
\begin{lemma}\label{lemma:omegas}
For  small enough $\varepsilon>0$ and all $n>N(k)$ there exist  $\omega_{n,j}=\frac{n^\varepsilon}{\log n}\rho_{n,j}$ with $0<\rho_{n,j}<2, \ j=1,...,k+1$, such that the moment of $\mu_{n,k}=\sigma_{n,k}+\gamma_{n,k}$ of order $2\ell$ equals $M_{2\ell}, \ell=0,1,...,k$.
\end{lemma}
\begin{proof}
Consider the system of $k+1$ linear equations in $k+1$ unknowns $\omega_{n,j}, j=1,...,k+1$.
\begin{multline}
M_{2\ell}=2\int_2^n\frac{C_k}{x^{2(k-\ell)+2}\log x}dx+2\sum_{j=1}^k\int_{2^{j-1}n}^{2^jn}\frac{C_k\omega_{n,j}}{x^{2(k-\ell)+2+\varepsilon}}dx\\
+2\int_{2^kn}^{\infty}\frac{C_k\omega_{n,k+1}}{x^{2(k-\ell)+2+\varepsilon}}dx, \ \ell=0,1,...,k.
\end{multline}
Recall the values of the following integrals ($\varepsilon, m>0, \ q \geq 0$):

\begin{multline}
\int_m^{2m} x^{-(2(q+1)+\varepsilon)}dx=\frac{1}{2q+1+\varepsilon}m^{-(2q+1+\varepsilon)}(1-(1/2)^{2q+1+\varepsilon}),\\
\int_m^{\infty} x^{-(2(q+1)+\varepsilon)}dx=\frac{1}{2q+1+\varepsilon}m^{-(2q+1+\varepsilon)}.
\end{multline}

 Let $q=k-\ell=0,1,...,k.$ Then  we get the following system of $k+1$ equations in unknowns 
$\rho_{n,j}, j=1,...,k+1$: 
\begin{multline}\label{eq:mainsystem}
\big(\sum_{j=1}^k \rho_{n,j}(2^{-(j-1)(2q+1+\varepsilon)}-2^{-j(2q+1+\varepsilon)})\big)+\rho_{n,k+1}2^{-k(2q+1+\varepsilon)}=\\
(2q+1+\varepsilon) (n^{2q+1}\log n)\int_{n}^{\infty} \frac{1}{x^{2q+2}\log x}dx.
\end{multline}

Applying the estimates

\begin{multline*}
\frac{1}{(2q+1)n^{2q+1}\log n} \biggl(1-\frac{2}{(2q+1)\log n}\biggr) < \int_n^{\infty} \frac{1}{x^{2q+2}\log x} dx\\
< \frac{1}{(2q+1)n^{2q+1}\log n}
\end{multline*}

we see that the right-hand side of the $q$-th equation lies between  $\frac{2q+1+\varepsilon}{(2q+1)} \biggl(1-\frac{2}{(2q+1)\log n}\biggr)$ and  $\frac{2q+1+\varepsilon}{2q+1}$. 
Set the right-hand side to be $\frac{2q+1+\varepsilon}{2q+1}$ and introduce new variables 
\begin{eqnarray*}
x_1=\rho_{n,1}, \quad x_j=\rho_{n,j}-\rho_{n,j-1} \textrm{ for }j=2,\ldots,k+1.
\end{eqnarray*}
Then we have the following system of equations:
\begin{eqnarray*}
\sum_{j=1}^{k+1} 2^{-(j-1)(2q+1)}x_j = \frac{2q+1+\varepsilon}{2q+1}, \ q=0,...,k.
\end{eqnarray*}
The coefficient matrix of this system is a submatrix of a generalized Vandermonde matrix,  hence nonsingular.   Cramer's rule together with continuity of determinant imply that, for sufficiently small $\varepsilon >0$, $x_1$ is close to $1$ (in particular positive).  Subtracting the equation numbered $q+1$ from the $q$-th equation for $q=0,\ldots,k-1$ eliminates $x_1$.  The resulting smaller system is also uniquely solvable.  Since the right-hand side of this system is close to zero if $\varepsilon >0$, the solutions are  also close to zero.  So $\rho_{n,1}$ is close to $1$ and so are other  $\rho_{n,j}$ if $\varepsilon$ is small enough.  Again by continuity of determinant, if $n > N(k)$, then the solutions of the system with right-hand side $\frac{2q+1+\varepsilon}{(2q+1)} \biggl(1-\frac{2}{(2q+1)\log n}\biggr)$ are all positive and bounded above by, say, $2$. The same holds for every system with right-hand side being a convex combination of $\frac{2q+1+\varepsilon}{(2q+1)} \biggl(1-\frac{2}{(2q+1)\log n}\biggr)$ and $\frac{2q+1+\varepsilon}{2q+1}$ . In particular we have  solutions $\rho_{n,j}$ of the system (\ref{eq:mainsystem}) satisfying $0<\rho_{n,j}<2,\  j =1,...,k+1$ for $\varepsilon>0$ small and all $n\gg 1$. 

\end{proof}

\begin{theorem}\label{theorem:conv2k+1}
For $\mu_k, \mu_{n,k}$ as above we have 
\[
\lim_{n\to\infty}\sup_{t\ne 0}\bigl|\frac{\varphi_{n,k}(t)-\varphi_(t)}{t^{2k+1}}\bigr|=0.
\]
\end{theorem}

\begin{proof}
Because of equality of moments of orders up to $s-1=2k$, the expression $\sup_{t\ne 0}\bigl|\frac{\varphi_{n,k}(t)-\varphi_k(t)}{t^{2k+1}}\bigr|$ is finite for each $n >N(k)$. Applying Taylor's formula with remainder in integral form we have 
\[
\varphi_{n,k}(t)-\varphi_k(t)=\int_0^1 \frac{(1-\lambda)^{2k-1}}{(2k-1)!}(\varphi_{n,k}^{(2k)}(\lambda t)-\varphi_k^{(2k)}(\lambda t)) t^{2k} d\lambda
\]
Hence 
\[
\sup_{t\ne 0}\biggl|\frac{\varphi_{n,k}(t)-\varphi_k(t)}{t^{2k+1}}\biggr| \leq \frac{1}{(2k-1)!}\sup_{\theta\ne 0}\biggl|\frac{\varphi_{n,k}^{(2k)}(\theta)-\varphi_k^{(2k)}(\theta)}{\theta}\biggr|.
\]

Reverting to the variable $t$, it is thus enough to prove that $$\lim_{n\to\infty}\sup_{t\ne 0}\biggl|\frac{\varphi_{n,k}^{(2k)}(t)-\varphi_k^{(2k)}(t)}{2C_kt}\biggr|=0.$$ 

Let $\varphi_{n,k}^{(2k)}(t)=f_{k,n}(t)+g_{n,k}(t)$ and $\varphi_k^{(2k)}(t)=f_{k,n}(t)+h_{k,n}(t)$, where $f_{k,n}(t)={\hat \sigma_{n,k}}^{(2k)}(t)$, $g_{k,n}(t)={\hat \gamma_{n,k}}^{(2k)}(t)$,   $h_{k,n}(t)={\hat \tau_{n,k}}^{(2k)}(t)$, and $\tau_{n,k}:=\mu_k-\sigma_{n,k}$. Recall that  $\mu_{n,k}$ and $\mu_k$ all have the same moment $M_{2k}$ of order $2k$ when $n \geq N(k)$. This gives $g_{n,k}(0)=h_{n,k}(0)$ for all $n \geq N(k)$.  Then

\begin{multline}
|\varphi_{n,k}^{(2k)}(t)-\varphi_k^{(2k)}(t)|=|g_{n,k}(t)-h_{n,k}(t)|=\\
|(g_{n,k}(t)-g_{n,k}(0))-(h_{n,k}(t)-h_{n,k}(0))| \le |g_{n,k}(0)-g_{n,k}(t)|+|h_{n,k}(0)-h_{n,k}(t)|.
\end{multline}

Due to symmetry, it is enough to assume that $t>0$.
Observe that

\begin{eqnarray}\label{eq:hfunc}
0\le \frac{h_{n,k}(0)-h_{n,k}(t)}{2C_kt}=\int_n^{\infty}\frac{1-\cos tx}{x^2\log x}dx < \frac{4}{\log n}
\end{eqnarray}

for $n \geq 4$ and all $t > 0$, which we already established in the previous section. For the functions $g_{n,k}$ we have the relation

\begin{multline*}
 \frac{g_{n,k}(0)-g_{n,k}(t)}{2C_k} =\sum_{j=1}^k \int_{2^{j-1}n}^{2^jn}\frac{(1-\cos tx)\omega_{n,j}}{x^{2+\varepsilon}}dx\\
+\int_{2^kn}^{\infty}\frac{(1-\cos tx)\omega_{n,k+1}}{x^{2+\varepsilon}}dx\\
\leq 2\frac{n^\varepsilon}{\log n} \int_n^\infty \frac{1-\cos tx}{x^{2+\varepsilon}}dx, 
\end{multline*}

since the coefficients $\omega_{j,n}$ satisfy $0 < \omega_{j,n} < \frac{2n^\varepsilon}{\log n}, \ j=1,...,k+1$. Proceeding as in the previous section, we have for $t \geq 1/n$ 

\begin{equation}\label{eq:gfunc1}
\frac{2n^\varepsilon}{\log n}\int_{n}^{\infty}\frac{(1-\cos tx)}{x^{2+\varepsilon}}dx < \frac{4}{n\log n} \leq \frac{4t}{\log n},
\end{equation}
while for  $0< t < 1/n$ we have 
\begin{multline}\label{eq:gfunc2}
\frac{2n^\varepsilon}{\log n}\int_{n}^{\infty}\frac{(1-\cos tx)}{x^{2+\varepsilon}}dx < \frac{2n^\varepsilon}{\log n}\biggl(\frac{t^{1+\varepsilon}}{1-\varepsilon}+2t^{1+\varepsilon}\biggr)\\
<\frac{8t}{\log n}
\end{multline}
when $\varepsilon < 1/2$. 
From the estimates (\ref{eq:hfunc}), (\ref{eq:gfunc1}) and (\ref{eq:gfunc2}) it follows that 
$$\lim_{n\to\infty}\sup_{t\ne 0}\biggl|\frac{\varphi_{n,k}^{(2k)}(t)-\varphi_k^{(2k)}(t)}{2C_kt}\biggr|=0,$$ hence also 
\[
\lim_{n\to\infty}\sup_{t\ne 0}\biggl|\frac{\varphi_{n,k}(t)-\varphi_k(t)}{t^{2k+1}}\biggr|=0.
\]
\end{proof}

The constructions allow us to establish the following:

\begin{theorem}\label{theorem: counterex}  Let $s =2k+1$  and let  $M_0,...,M_{2k}$ be the moments of the measure $\mu_k$ constructed in Subsection \ref{subsection:greater}. Let $N \geq 1$ and let 
$\mathcal{P}_s^=$ denote the space of all probability measures  $\nu$ in $\mathbb{R}^N$ for which $\int_{\mathbb{R}^N}v^\beta d\nu(v)= M_\beta, \ |\beta|=0,1,...,2k$ and $\int_{\mathbb{R}^N}v^\beta d\nu(v)=0, \ |\beta|=s$. The metric space $(\mathcal{P}_s^=, d_s)$ is not complete.
\end{theorem}

\begin{proof} When $N=1$, Theorem \ref{theorem:conv2k+1} yields that $d_s(\mu_{n,k},\mu) \to 0$ as $n \to \infty$ (in particular, the sequence $\mu_{n,k}$ is Cauchy in $\mathcal{P}_s$) with $\mu_{n,k} \in \mathcal{P}_s^=$, but $\mu \not \in \mathcal{P}_s^=$. When $N >1$,   counterexamples can also be obtained. Let  $s=2k+1, k \geq 1$. Introduce the measures $\nu_{n,k}$ and $\nu_k$ with characteristic functions respectively $\psi_{n,k}(t_1,...,t_N)=\varphi_{n,k}(t_1)...\varphi_{n,k}(t_N)$ and $\psi(t_1,...,t_N)=\varphi(t_1)...\varphi(t_N)$ with $\varphi_{n,k}, \varphi_n$ as in  Theorem \ref{theorem:conv2k+1}, $n \gg 1$. Define the space $\mathcal{P}_s^=$ as the space of all probability measures   in $\mathbb{R}^N$ whose moment of order $\beta$ agrees with the moment of order $\beta$ of $\nu_{n,k}$ for all $\beta$ such that $|\beta|\leq s$. This space is not complete in the metric $d_s$. Take  e.g. $\beta=(s,0,...,0)$. Then 
\[
\frac{|\psi_{n,k}(t)-\psi_k(t)|}{\|t\|^s} \leq c(N) \frac{|\varphi_{n,k}(t_1)-\varphi_k(t_1)|}{|t_1|^s},
\]
so $d_s(\nu_{n,k},\nu_k) \to 0$ as $n \to \infty$, but $\nu_k$ does not have the moment of order $(s,0,...,0)$.
\end{proof}

\textbf{Acknowledgments:} I thank the anonymous referee(s) for suggesting the idea for Example \ref{example:Cauchy},  for bringing to my attentions the references \cite{BG} and \cite{TT} and for pointing out some shortcomings of a previous version of this note. I also thank Shmuel Friedland for valuable discussions of problems related to Section 4. Some results of this note were presented at the conference ``Geometry of Banach Spaces and related topics'',  Krak\'ow, June 8-10, 2017, in honor of Professor Grzegorz Lewicki. I thank the organizers for the invitation and to participants of the seminar ``Methods of Approximation Theory'' at Jagiellonian University, Krak\'ow, for useful feedback. Finally, I thank  my employer American Mathematical Society for supporting my study leave in Krak\'ow in February--June 2017 and the Chair of Approximation Theory of the Institute of Mathematics at Jagiellonian University (in particular Marta Kosek, Leokadia Bia\l as-Cie\.z and Grzegorz Lewicki) for additional organizational and financial support.

\end{document}